\documentclass[graybox]{svmult}

\usepackage{type1cm}        % activate if the above 3 fonts are
\usepackage{makeidx}         % allows index generation
\usepackage{graphicx}        % standard LaTeX graphics tool
\usepackage{multicol}        % used for the two-column index
\usepackage[bottom]{footmisc}% places footnotes at page bottom

\usepackage{newtxtext}       % 
\usepackage[varvw]{newtxmath}       % selects Times Roman as basic font

\usepackage{graphicx}       % standard LaTeX graphics tool
\usepackage[utf8]{inputenc}
\usepackage[T1]{fontenc}
\usepackage{microtype} % good font tricks

\usepackage{cite}
\usepackage{array,colortbl}

\DeclareFontFamily{U}{mathx}{\hyphenchar\font45}
\DeclareFontShape{U}{mathx}{m}{n}{
<5> <6> <7> <8> <9> <10>
<10.95> <12> <14.4> <17.28> <20.74> <24.88>
mathx10
}{}
\DeclareSymbolFont{mathx}{U}{mathx}{m}{n}
\DeclareFontSubstitution{U}{mathx}{m}{n}
\DeclareMathAccent{\widecheck}{0}{mathx}{"71}

\def\citep#1#2{\cite[{#1}]{#2}}

\newcommand{\RefEq}[1]{~\textup{(\ref{#1})}}

\newcommand{\RefSec}[1]{Section~\textup{\ref{#1}}}

\newcommand{\RefLem}[1]{Lemma~\textup{\ref{#1}}}

\newcommand{\RefAlg}[1]{Algorithm~\textup{\ref{#1}}}
\newcommand{\RefFig}[1]{Figure~\textup{\ref{#1}}}

\usepackage[linesnumbered,ruled,vlined]{algorithm2e}
\SetKwInput{KwInput}{Input}                % Set the Input
\SetKwInput{KwOutput}{Output}              % set the Output

\usepackage{caption}
\usepackage{subcaption}

\newcommand{\Real}{{\mathbb{R}}}
\newcommand{\Integer}{{\mathbb{Z}}}
\newcommand{\bsX}{{\boldsymbol{X}}}
\newcommand{\bsY}{{\boldsymbol{Y}}}
\newcommand{\bsy}{{\boldsymbol{y}}}
\newcommand{\calF}{{\mathcal{F}}}
\newcommand{\calN}{{\mathcal{N}}}
\newcommand{\bseta}{{\boldsymbol{\eta}}}
\newcommand{\bstheta}{{\boldsymbol{\theta}}}
\newcommand{\calL}{{\mathcal{L}}}
\newcommand{\bbI}{{\mathbb{I}}}
\newcommand{\calU}{{\mathcal{U}}}

\makeindex             % used for the subject index
\begin{document}

\title*{Multi-fidelity No-U-Turn Sampling}
\author{Kislaya Ravi, Tobias Neckel and Hans-Joachim Bungartz}
\institute{Kislaya Ravi \at Technical University of Munich, Garching, Germany. \email{kislaya@cit.tum.de}
\and Tobias Neckel \at Technical University of Munich, Garching, Germany. \email{neckel@cit.tum.de}
\and Hans-Joachim Bungartz \at Technical University of Munich, Garching, Germany. \email{bungartz@cit.tum.de}}

\newcommand{\dram}[1]{\textcolor{red}{#1}}

\maketitle

\abstract{
    Markov Chain Monte Carlo (MCMC) methods often take many iterations to converge for highly correlated or high-dimensional target density functions. Methods such as Hamiltonian Monte Carlo (HMC) or No-U-Turn Sampling (NUTS) use the first-order derivative of the density function to tackle the aforementioned issues. However, the calculation of the derivative represents a bottleneck for computationally expensive models. We propose to first build a multi-fidelity Gaussian Process (GP) surrogate. The building block of the multi-fidelity surrogate is a hierarchy of models of decreasing approximation error and increasing computational cost. Then the generated multi-fidelity surrogate is used to approximate the derivative.  The majority of the computation is assigned to the cheap models thereby reducing the overall computational cost. The derivative of the multi-fidelity method is used to explore the target density function and generate proposals. We select or reject the proposals using the Metropolis Hasting criterion using the highest fidelity model which ensures that the proposed method is ergodic with respect to the highest fidelity density function. We apply the proposed method to three test cases including some well-known benchmarks to compare it with existing methods and show that multi-fidelity No-U-turn sampling outperforms other methods. 
}

\section{Introduction}
Bayesian inference is a widely used method in many fields such as astrophysics, geological exploration, machine learning etc \cite{ravi-bib-kaipio}. Exact inference of the posterior is rarely possible. The Markov Chain Monte Carlo (MCMC) method is one of the most commonly used methods in Bayesian inference to draw samples from the posterior. MCMC draws samples from the target density function that converge to the given distribution and are asymptotically unbiased. Classical methods such as the Metropolis-Hastings algorithm \cite{ravi-bib-metropolis} and Gibbs sampling \cite{ravi-bib-gibbs} often take a very long time to converge because of the inefficient random walk. This makes it intractable for computationally expensive models.

Hamilton Monte Carlo (HMC) \cite{ravi-bib-neal-handbook} also known as hybrid Monte Carlo is a method that circumvents the random walk by using the derivative of the target density to propose better samples. There are two disadvantages of HMC. First, the user needs to specify some parameters of the algorithm manually. A poor choice of these parameters will have a considerable negative impact on the quality of drawn samples. Second, HMC requires the values of the derivative to be evaluated multiple times to propose one sample which means multiple evaluations of the target density function making it infeasible for computationally expensive models. The first issue can be addressed using No-U Turn sampling (NUTS) \cite{ravi-bib-nuts} which automatically determines the values of the crucial parameters. However, NUTS still requires access to the derivative.

In this work, we propose to replace the use of the actual model for derivative evaluation with a surrogate. The proposed sample is accepted/rejected based on the evaluation of the actual model. This ensures that the samples are invariant with respect to the target density function and not it's surrogate. Moreover, we also build the surrogate using multi-fidelity method.

One is often provided with different types of models solving the same phenomenon. We denote computationally cheap but inaccurate models as low-fidelity models whereas the expensive but accurate models are called high-fidelity models. We can combine these models efficiently for different applications by using the low-fidelity models more often than the high-fidelity function. Such methods fall under the category of multi-fidelity methods \cite{ravi-bib-survey-mf}. In this work, we use the concept of Gaussian Processes \cite{ravi-bib-rasmussen-book} to build multi-fidelity surrogates using a non-linear fusion of models \cite{ravi-bib-perdikaris-nargp, ravi-bib-lee-mfgp}. To the best of our knowledge, the algorithm proposed in this work is the first of its kind that applies multi-fidelity to NUTS. The algorithm is implemented in python and the implementation is publicly accessible \textsuperscript{1}. \footnotetext[1]{\url{https://github.com/KislayaRavi/MuDaFuGP}} This library can be easily coupled with solvers from various fields of application to generate multi-fidelity surrogates and solving computationally expensive problems such as forward uncertainty quantification, Bayesian optimization and MCMC sampling. In this work, we only discuss the MCMC sampling part of the implementation.

We first introduce the Bayesian inference setup in \RefSec{ravi-sec:bayesian-inverse}. Then, we provide a brief theoretical background for HMC and NUTS in \RefSec{ravi-sec:mcmc}. After that, we discuss the multi-fidelity method in \RefSec{ravi-sec:mf} where we explain the multi-fidelity Gaussian Process surrogate, delayed acceptance algorithm and finally describe the proposed algorithm by combining all the previous methods. Finally, we test the performance of the proposed method and compare it with some of the commonly used sampling methods in \RefSec{ravi-sec:results} and summarise our conclusions in \RefSec{ravi-sec:conclusion}.

\section{Bayesian Inference}
\label{ravi-sec:bayesian-inverse}
Let $\bsX \in \Real^d$ denote the parameter space and $\bsY \in \Real^m$ a separable Banach space that represents the data space. $d, m \in \Integer^+$ are the dimensions of parameters and data respectively. Both the parameters and the data belong to the finite-dimensional spaces which allow us to work with densities with respect to the Lebesque measure.
Let us consider a function $\calF : \bsX \rightarrow \bsY$, which is a map from the parameter space to the data space.

The noisy observations $\bsy \in \bsY$ are typically modeled by adding some Gaussian noise $\bseta \sim \calN(0, \Gamma)$, where $\Gamma$ is a positive definite covariance matrix. For a given parameter $\bstheta \in \bsX$,  we can express observed data $\bsy$ as a random variable: 
\begin{equation}
    \bsy = \calF(\bstheta) + \bseta
    \label{ravi-eq:basic-inverse}
\end{equation}
In inverse problems, the target is to solve \RefEq{ravi-eq:basic-inverse} w.r.t.~$\bstheta$, given the observations $\bsy$, to infer the true parameters $\bstheta^{true} \in \bsX$. However, this problem is ill-posed in the sense of Hadamard \cite{ravi-bib-hadamard}. One can cure the ill-posedness of the problem by reformulating it as a Bayesian inference problem \cite{ravi-bib-kaipio}. 

We consider the target $\bstheta$ as a random variable that follows a prior distribution $p(\bstheta)$ on $\bsX$ which represents our assumption on the parameters without looking at the observations. Let us assume that $\bstheta$ is independent of the noise $\bseta$. The likelihood $p(\bsy|\bstheta)$ represents the quality of an assumed $\bstheta$ to produce the given observations $\bsy$. Because of Gaussian noise $\bseta$, the likelihood can be written as:
\begin{equation}
    p(\bsy|\bstheta) \propto \exp \left( -\frac{1}{2} || \Gamma^{-1/2}(\bsy - \calF(\bstheta)) ||^2 \right)
    \label{ravi-eq:likelihood-def}
\end{equation}
In Bayesian Inference, we want to find the posterior distribution $p(\bstheta | y)$, which represents the probability distribution of the parameter given the observation. Using Bayes theorem, one obtains
\begin{equation}
    p(\bstheta|\bsy) = \frac{p(\bsy|\bstheta)p(\bstheta)}{p(\bsy)} \ .
    \label{ravi-eq:bayes}
\end{equation}

The denominator term in \RefEq{ravi-eq:bayes} is known as evidence which is often difficult to calculate. However, it is independent of $\bstheta$. There are multiple ways to draw samples from the posterior distribution. One of the commonly used methods is the Markov Chain Monte Carlo (MCMC) algorithm \cite{ravi-bib-neal-handbook} which also circumvents the evidence term. In this work, we deal with No-U Turns Sampling (NUTS) \cite{ravi-bib-nuts} which we explain in the next section. 

\section{Hamilton Monte Carlo and No-U-Turn Sampling}
\label{ravi-sec:mcmc}
No-U-Turn Sampling (NUTS) is a sampling method based on the Hamilton Monte Carlo (HMC) \cite{ravi-bib-neal-handbook} method. HMC utilizes the geometry of the density function to propose better samples. We introduce an extra momentum variable $r \in \Real^d$. The Hamiltonian of the system is defined using the state ($\theta$) and the momentum term ($r$). Let $\calL (\theta)$ be the logarithm of the target density function $\pi (\theta)$. With $r \cdot r$ representing the inner product of $\Real^d$, the Hamiltonian of the system is defined as

\begin{equation}
    H(\theta, r) = - \calL(\theta) + \frac{1}{2} r \cdot r \ .
    \label{ravi-eq:hamiltonian}
\end{equation} 
The canonical distribution of the system is
\begin{equation}
    p(\theta, r) \propto \exp \{ -H(\theta, r) \} \ .
    \label{ravi-eq:canonical-density}
\end{equation}
At each iteration, the momentum term is resampled ($r \sim \calN(0, \bbI_d)$) to propose the next state by solving the Hamiltonian dynamics for some time-steps. The Hamiltonian dynamics are numerically simulated using a time-stepping method which is reversible and volume-preserving. The leapfrog method is one of the most commonly used methods for this purpose and can be formulated as
\begin{equation}
    \begin{aligned}
        r^{i+1/2} &= r^i + (\epsilon/2) \nabla_{\theta} \calL (\theta^i) \\
        \theta^{i+1} &= \theta^i + \epsilon r^{i+1/2} \\
        r^{i+1} &= r^{i+1/2} + (\epsilon/2) \nabla_{\theta} \calL (\theta^{i+1}) \ ,
    \end{aligned}
    \label{ravi-eq:leapfrog-hmc}
\end{equation}
where $\epsilon$ is the time-step size for the leapfrog method.
We perform the fictitious time integration for given $T$ steps to obtain the proposal ($\tilde{\theta}, \tilde{r}$). The proposal is accepted/rejected based on the acceptance probability $\alpha$ defined as:
\begin{equation}
    \alpha = \min \left[ 1, \frac{\exp \{ -H(\tilde{\theta}, \tilde{r}) \}}{\exp \{ -H(\theta, r) \}}  \right]
    = \min \left[ 0, H(\theta, r) - H(\tilde{\theta}, \tilde{r}) \right]
    \label{ravi-eq:alpha-hmc}
\end{equation}
Note that the normalization term will get canceled out. So, we can sample from density functions that are not normalized. 
\begin{lemma}
    The HMC algorithm is ergodic with respect to the canonical density function mentioned in \RefEq{ravi-eq:canonical-density} provided the leapfrog integrator does not generate periodic proposals.
\end{lemma}
\begin{proof}
    Interested readers can refer to \cite{ravi-bib-neal-handbook} for more detailed proof. If the leapfrog integrator does not generate periodic proposals then the Markov chain does not get trapped in certain subsets. We can also prove the detailed balance of the algorithm by taking into account the volume-preserving property of the leapfrog method. 
\end{proof}
The quality of the samples depends on the choice of the tuning parameters, namely the number of steps $T$ and the step size $\epsilon$. NUTS is an extension of HMC, where the need to specify a fixed value of $T$ is eliminated by performing the time integration until one observes a U-turn. After the U-turn, the time integrator is very likely to retrace the old steps or visit states that are very close to already explored states. The slice sampling method is used to sample from the canonical distribution $p(\theta, r)$. The sign of the momentum term does not affect the value of $p(\theta, r)$. So, we build a tree while running the time integration. The first step is to sample the momentum ($r \sim \calN(0, \bbI_d)$) as in standard HMC. We set the depth of the tree to zero ($l=0$). At every level, we perform the time integration for $2^l$ steps to obtain ($\tilde{\theta}, \tilde{r}$) and increase the value of $l$ by 1. At the end of time integration, we randomly draw either $1$ or $-1$ with equal probability. If the number is $1$, then we move to the rightmost node of the tree ($\theta^{r}$) and perform the next time integration steps using the corresponding momentum term. If the number is $-1$, then we move to the leftmost node of the tree ($\theta^{l}$) and perform the next time integration steps using the negative of the corresponding momentum term. At the end of every step, we check if we observe a U-turn. A U-turn is checked by verifying the following condition 
\begin{equation}
    (\theta^r - \theta^l) \cdot \tilde{r} < 0 \ .
    \label{ravi-eq:u-turn}
\end{equation}
We cut a random slice out of the canonical distribution $u \sim \calU[0, p(\theta, r)]$. 
Out of all the states explored while building the tree, we select the states that satisfy the condition $p(\tilde{\theta}, \tilde{r}) \geq u$.  We uniformly select one state out of all the chosen states which serves as the next state $\theta^{i+1}$. We repeat the aforementioned steps until the required number of samples are obtained. Interested readers can refer to \cite{ravi-bib-nuts} for the detailed algorithm. The samples drawn using NUTS are ergodic with respect to the canonical density function \cite{ravi-bib-nuts}.  

One of the main challenges of any gradient-based sampling method is the calculation of the gradient. In many cases, the gradients are not available but one can use the numerical approximation of the gradients. With the help of auto differentiation methods, the derivatives are easier to calculate. However, we need to compute the derivative at every point of the time integration step. This is very resource-intensive for computationally expensive functions. In this work, we suggest building a computationally cheap multi-fidelity surrogate of the function to approximate the derivative. In the next section, we explain the multi-fidelity method in more detail.

\section{Multi-fidelity Methods}
\label{ravi-sec:mf}
For different applications, we often have a hierarchy of models. These functions model the same phenomena but with different levels of accuracy and resource requirements. Low-fidelity functions are computationally inexpensive but have relatively high errors. In contrast, high-fidelity functions are computationally expensive but accurately model the phenomena. We arrange the functions in increasing order of fidelity  $\{ f_1, f_2, ..., f_L \}$, where $f_i:\Real^d \rightarrow \Real ,  \, i=1,2,...,L$. We need to evaluate the functions multiple times in applications that require uncertainty quantification, optimization, inverse problems, etc. Computing high-fidelity frequently is practically not possible because of resource limitations. Moreover, we cannot completely replace the high-fidelity function with the low-fidelity function because we need the final result to be reasonably accurate, too. Multi-fidelity methods combine different fidelity functions to leverage the speed of the low-fidelity function and the accuracy of the high-fidelity method. A detailed survey of the different multi-fidelity methods is given in \cite{ravi-bib-survey-mf}. In this work, we use Gaussian Processes (GPs) to create the multi-fidelity surrogate.

We discuss a method to use GP to build a multi-fidelity surrogate in \RefSec{ravi-subsec:mfgp}. The surrogate guides NUTS to propose samples. In \RefSec{ravi-subsec:da} we discuss the delayed acceptance method that ensures ergodicity of samples with respect to the high-fidelity function. Finally, in \RefSec{ravi-subsec:mfnuts} we combine all the different components to explain the multi-fidelity No-U Turn sampling method.
\vspace{-1.5em}
\subsection{Multi-fidelity Gaussian Process Surrogates}
\label{ravi-subsec:mfgp}
Let us consider a two-fidelity system where $f_l(\theta)$ and $f_h(\theta)$ represent the low and high-fidelity functions, respectively, that reside on the same parameter space with points $\theta \in \Real^d$. We perform a non-linear fusion of the different fidelities as discussed in \cite{ravi-bib-perdikaris-nargp,ravi-bib-lee-mfgp}. We first discuss non-linear auto-regressive Gaussian process (NARGP) \cite{ravi-bib-perdikaris-nargp} where the high-fidelity model is expressed as a function of the low-fidelity model and the parameter space as:
\begin{equation}
    f_h(\theta) = g(f_l(\theta), \theta) \ , 
    \label{ravi-eq:mfgp-nargp-ansatz}
\end{equation}
where $g$ is a function that resides in a $d+1$ dimensional space. In this way, one can explore the non-linear relationship between the high-fidelity and low-fidelity function. A Gaussian process (GP) \cite{ravi-bib-rasmussen-book} is used to express the function $g$ because it provides not just the prediction but also the confidence interval of the predicted value. The kernel of NARGP is expressed in a way to mimic the auto-regressive structure as discussed in \cite{ravi-bib-kennedy-ohagan-auto-reg} as:
\begin{equation}
    k_{NARGP} = k_f(f_l(\theta), f_l(\theta')|\lambda_{f}) k_{\rho}(\theta, \theta'|\lambda_{\rho}) + k_{\delta}(\theta, \theta'|\lambda_{\delta}) \ , 
    \label{ravi-eq:nargp-kernel}
\end{equation}
where $k_f, k_{\rho}, k_{\delta}$ are positive definite kernels and $\lambda_f, \lambda_{\rho}, \lambda_{\delta}$ are the corresponding hyperparameters. We use square exponential kernels in our work. However, one can choose a tailored kernel depending on the application case \cite{ravi-bib-rasmussen-book}. 

One can also include the derivative of the low-fidelity function in the formulation of the composite function $g$ in \RefEq{ravi-eq:mfgp-nargp-ansatz}. However, the derivative of the low-fidelity function may not be readily available. Instead of calculating the derivative using the finite difference method that can sometimes cause rounding errors, extra parameters are added in the composite function that corresponds to the lag term $f_l(\theta-\tau)$ and $f_l(\theta+\tau)$ which will mimic the derivative of the low-fidelity function, where $\tau$ represent a small real number. Now, the structure of the surrogate is:
\begin{equation}
    f_h(\theta) = g(f_l(\theta), f_l(\theta+\tau), f_l(\theta-\tau), \theta)
    \label{ravi-eq:mfgp-gpdf-ansatz}
\end{equation}
We can use any kernel of our choice to build the surrogate. It does not need to follow a structure like NARGP. This multi-fidelity Gaussian process surrogate is discussed in \cite{ravi-bib-lee-mfgp}. We call this method Gaussian Process with Derivative Fusion (GPDF). 
% We also implemented a kernel with structure like NARGP with lag terms. We call that surrogate as Gaussian Process with Derivatiove Fusion and Composite kernel (GPDFC).

\subsection{Delayed acceptance}
\label{ravi-subsec:da}
The Delayed Acceptance (DA) algorithm \cite{ravi-bib-delayed-acceptance} was developed to sample from a distribution $\pi(\theta)$ when there exists an approximation $\pi^*(\theta)$ of the target density distribution. The distribution does not need to be normalized. Just like any Metropolis-Hastings algorithm, one needs a proposal density function $q(\tilde{\theta}|\theta)$ and an initial point $\theta^0$. The steps of the simplified version of DA are summarized below:
\begin{enumerate}
    \item Draw a sample from the proposal distribution $\tilde{\theta} \sim q(\tilde{\theta}|\theta)$ and accept/reject the proposal for the next step based on the following acceptance probability:
    \begin{equation}
        \alpha^*(\tilde{\theta} | \theta) = \min \left\{ 1, \frac{\pi^*(\tilde{\theta}) q(\theta|\tilde{\theta})}{\pi^*(\theta) q(\tilde{\theta}|\theta)} \right\}
        \label{ravi-eq:da1}
    \end{equation}
    \item If the sample is accepted in the previous step, then accept/reject the proposal $\tilde{\theta}$ based on the following acceptance probability:
    \begin{equation}
        \begin{aligned}
            q^*(\tilde{\theta}|\theta) &= \alpha^*(\tilde{\theta} | \theta)q(\tilde{\theta}|\theta) \\
            \alpha(\tilde{\theta} | \theta) &= \min \left\{ 1, \frac{q^*(\theta|\tilde{\theta)} \pi(\tilde{\theta})}{q^*(\tilde{\theta}|\theta) \pi(\theta)} \right\}
        \end{aligned}
        \label{ravi-eq:da2}
    \end{equation}
    \item The next sample is taken to be the same as the previous sample if the proposal is rejected at any of the previous two steps.
    \item Repeat the first three steps until the required number of samples are drawn. 
\end{enumerate}
\begin{lemma}
    The DA algorithm preserves the detailed balance with respect to the target density $\pi(\theta)$.
    \label{ravi-lemma-detailed-balance-da}
\end{lemma}
\begin{proof}
    We need to show $q(\tilde{\theta}|\theta) \alpha^*(\tilde{\theta}|\theta) \alpha(\tilde{\theta}|\theta) \pi(\theta) = q(\theta|\tilde{\theta}) \alpha^*(\theta|\tilde{\theta}) \alpha(\theta|\tilde{\theta}) \pi(\tilde{\theta})$ to satisfy the detailed balance.
    The proposed sample is either accepted or rejected. Proving the detailed balance for rejected case ($\tilde{\theta} = \theta$) is trivial. Let us analyze the case when the sample is accepted. We look at the left-hand side of the target equation. 
    \begin{align*}
        & q(\tilde{\theta}|\theta) \alpha^*(\tilde{\theta}|\theta) \alpha(\tilde{\theta}|\theta) \pi(\theta) \\
        &=q(\tilde{\theta}|\theta) \min \left\{ 1, \frac{\pi^*(\tilde{\theta}) q(\theta|\tilde{\theta})}{\pi^*(\theta) q(\tilde{\theta}|\theta)} \right\} \min \left\{ 1, \frac{q(\theta|\tilde{\theta}) \min \left\{ 1, \frac{\pi^*(\theta) q(\tilde{\theta}|\theta)}{\pi^*(\tilde{\theta}) q(\theta|\tilde{\theta})} \right\} \pi(\tilde{\theta})}{q(\tilde{\theta}|\theta)\min \left\{ 1, \frac{\pi^*(\tilde{\theta}) q(\theta|\tilde{\theta})}{\pi^*(\theta) q(\tilde{\theta}|\theta)}  \right\} \pi(\theta)} \right\} \pi(\theta) \\
        &=q(\tilde{\theta}|\theta) \min \left\{ 1, \frac{\pi^*(\tilde{\theta}) q(\theta|\tilde{\theta})}{\pi^*(\theta) q(\tilde{\theta}|\theta)} \right\} \min \left\{ 1, \frac{q(\theta|\tilde{\theta}) \min \left\{ 1, \frac{\pi^*(\theta) q(\tilde{\theta}|\theta)}{\pi^*(\tilde{\theta}) q(\theta|\tilde{\theta})} \right\} \pi(\tilde{\theta}) \pi^*(\tilde{\theta}) }{q(\tilde{\theta}|\theta)\min \left\{ 1, \frac{\pi^*(\tilde{\theta}) q(\theta|\tilde{\theta})}{\pi^*(\theta) q(\tilde{\theta}|\theta)}  \right\} \pi(\theta) \pi^*(\tilde{\theta}) } \right\} \pi(\theta) \frac{\pi^*(\theta)}{\pi^*(\theta)} \\
        &= \min \left\{ \pi^*(\theta) q(\tilde{\theta}|\theta), \pi^*(\tilde{\theta}) q(\theta|\tilde{\theta}) \right\} \min \left\{ \frac{\pi(\theta)}{\pi^*(\theta)}, \frac{\min \left\{q(\theta|\tilde{\theta})\pi^*(\tilde{\theta}), q(\tilde{\theta}|\theta) \pi^*(\theta) \right\} \pi (\tilde{\theta})}{\min \left\{q(\tilde{\theta}|\theta) \pi^*(\theta), q(\theta|\tilde{\theta}) \pi^*(\theta) \right\} \pi^*(\tilde{\theta})} \right\} \\
        &= \min \left\{ \pi^*(\theta) q(\tilde{\theta}|\theta), \pi^*(\tilde{\theta}) q(\theta|\tilde{\theta}) \right\} \min \left\{ \frac{\pi(\theta)}{\pi^*(\theta)}, \frac{\pi(\tilde{\theta})}{\pi^*(\tilde{\theta})} \right\}
    \end{align*}
    We observe that the final expression is symmetric about $\theta$ and $\tilde{\theta}$. So, we will arrive at the same expression if we simplified $q(\theta|\tilde{\theta}) \alpha^*(\theta|\tilde{\theta}) \alpha(\theta|\tilde{\theta}) \pi(\tilde{\theta})$ which is the right-hand side of the detailed balance equation. 
    This completes the proof.
\end{proof}
% In this work, we propose samples using the surrogate and employ DA algorithm to select samples ensuring that the detailed balance is conserved with respect to the high fidelity function.

\vspace{-1.5em}
\subsection{Multi-fidelity No-U Turn sampling}
\label{ravi-subsec:mfnuts}
The steps of the Multi-fidelity No-U Turn sampling (MFNUTS) are shown in \RefAlg{ravi-alg:mfnuts}. We can broadly divide the algorithm into two parts, namely the offline stage and the sampling stage. The offline stage includes building the surrogate and the sampling stage involves using the surrogate and the high-fidelity function to obtain samples. The first part of the algorithm is to build a Multi-fidelity Gaussian Process (MFGP) surrogate as described in \RefSec{ravi-subsec:mfgp}. We build the surrogate of the logarithm of the target density function and randomly sample some points within the design space to build all the surrogates using NARGP and GPDF. Then, we select the model with the least mean squared error with respect to some test points. The surrogate is represented by $\calL_s(\theta)$ and the corresponding canonical distribution as shown in \RefEq{ravi-eq:canonical-density} by $p_s(\theta, r)$.

\begin{algorithm}[!t]
    \KwInput{$\calF := \{f_1, f_2, ..., f_L\}; \pi_L(\theta); \theta^{\text{upper}}; \theta^{\text{lower}}; \theta^0; M^{\text{adapt}}; M^{\text{samples}}$}
    \KwOutput{$\{\theta^1, \theta^2, ..., \theta^{M^{\text{samples}}}\}$}

    \tcc{Build MFGP surrogate}
    $f_s(\theta) \leftarrow $ \texttt{BuildMFGP}$(\calF, \theta^{\text{upper}}, \theta^{\text{lower}})$ \\
    Define log likelihood for surrogate ($\calL_s(\theta)$) and corresponding density function ($\pi_s(\theta):=\exp \{\calL_s(\theta)\}$) \\
    Define the canonical distribution of surrogate $p_s(\theta, r) := \exp\{ -\calL_s(\theta) + \frac{1}{2} r \cdot r \}$, where $r$ is the momentum term\\
    
    \tcc{Perform Adaptation steps using the surrogate to find step size}
    $\epsilon \leftarrow $ \texttt{FindStepSize}$(\calL_s, \theta^0, M^{\text{adapt}})$

    \tcc{Sampling step}
    \For{$i = 1$ to $M^{\text{samples}}$}{
        \tcc{Propose one sample by running one iteration of NUTS using the surrogate}
        $(\tilde{\theta}, \tilde{r}) \leftarrow $ \texttt{NUTS} $(\theta^{i-1}, \epsilon)$ \\
        \tcc{Accept or reject the proposal using DA algorithm}
        \If{$\tilde{\theta} \neq \theta^{i-1}$}{
            $\alpha_{\text{MFNUTS}}(\tilde{\theta}|\theta) = \min \left\{ 1, \frac{\min \left\{ 1, \frac{p_s(\theta, r)}{p_s(\tilde{\theta}, \tilde{r})} \right\}\pi_L(\tilde{\theta})}{\min \left\{ 1, \frac{p_s(\tilde{\theta}, \tilde{r})}{p_s(\theta, r)} \right\}\pi(_L\theta)} \right\}$\\
            \If{$\alpha_{\text{MFNUTS}}(\tilde{\theta}|\theta) > \calU [0, 1]$}{
                $\theta^i = \tilde{\theta}$
            }
            \Else{
                $\theta^i = \theta^{i-1}$
            }
        }
        \Else{
            $\theta^i = \theta^{i-1}$
        }
    }
    \caption{Multi-fidelity No-U Turn Sampling(MFNUTS)}
    \label{ravi-alg:mfnuts}
\end{algorithm}

The second step is to determine the step size ($\epsilon$) of the leapfrog method. On the one hand, a big value of $\epsilon$ will cause the leapfrog method to visit a few states before a U-turn is observed. This corresponds to an improper exploration of states and may lead to many rejected samples. On the other hand, a very small value of $\epsilon$ results in the exploration of many nearby states before reaching the stopping criterion which is inefficient and leads to a too high acceptance ratio. So, we optimize the value of $\epsilon$ such that the expected value of the acceptance ratio reaches a target value ($\delta$). In this work we use $\delta=0.65$ following the argument from \cite{ravi-bib-neal-handbook,ravi-bib-optimal-tuning-mcmc}. To achieve this, the dual averaging technique as described in \cite{ravi-bib-nuts,ravi-bib-nesterov-dual-averaging} is used. Only the surrogate canonical distribution $p_s(\theta, r)$ is used to determine the step size. We are given a starting point $\theta_0$ and we randomly sample a momentum term ($r_0 \sim \calN(0, \bbI_d)$). We start with some assumed value of $\epsilon$, perform one step of the leapfrog method, and check if the acceptance probability as shown in \eqref{ravi-eq:alpha-hmc} is greater than 0.5. The value of $\epsilon$ is halved until the criterion is satisfied. This is considered as the starting value of the step size ($\epsilon_0$) for the dual averaging algorithm which we run for some predefined steps. We call the part of finding the optimal value of the step size ($\epsilon$) the adaptation step.

We start the actual sampling step after the adaptation step. In each sampling step, we run one step of NUTS as described in \RefSec{ravi-sec:mcmc} using the surrogate canonical density function $p_s(\theta, r)$. Let the proposed sample be $(\tilde{\theta}, \tilde{r})$. The sample is accepted or rejected using the DA algorithm as described in \RefSec{ravi-subsec:da}. Let us represent the density function corresponding to the highest fidelity model as $\pi_L(\theta)$. We know that the leapfrog method is time-reversible. So, the proposal distribution become symmetric $q((\tilde{\theta}, \tilde{r})|(\theta, r)) = q((\theta, r)|(\tilde{\theta}, \tilde{r}))$. After substituting the previous assumption in\RefEq{ravi-eq:da2}, we obtain the acceptance ratio of the MFNUTS algorithm:

\begin{equation}
    \alpha_{\text{MFNUTS}}(\tilde{\theta}|\theta) = \min \left\{ 1, \frac{\min \left\{ 1, \frac{p_s(\theta, r)}{p_s(\tilde{\theta}, \tilde{r})} \right\}\pi_L(\tilde{\theta})}{\min \left\{ 1, \frac{p_s(\tilde{\theta}, \tilde{r})}{p_s(\theta, r)} \right\}\pi_L(\theta)} \right\} \ .
\end{equation}

Using \RefLem{ravi-lemma-detailed-balance-da}, we can show that the samples generated conserve the detailed balance with respect to $\pi_L(\theta)$. If the leapfrog steps do not get stuck in periodic cycles, then the samples are ergodic with respect to the density function of the highest-fidelity model.

\section{Numerical Results}
\label{ravi-sec:results}
We are going to compare the MFNUTS algorithm with the Metropolis-Hastings algorithm, HMC, NUTS, and the Delayed Rejection Adaptive Metropolis (DRAM) \cite{ravi-bib-dram} algorithm. The implementation of MFNUTS is available in the Github repository\textsuperscript{2}. We use Paramonte  \cite{ravi-bib-paramonte} to run DRAM. \footnotetext[2]{\url{https://github.com/KislayaRavi/MuDaFuGP}}Tensorflow Probability \cite{ravi-bib-tfd} is used to run the other three algorithms. We use three cases to compare the different methods and draw 10,000 samples for each case with 2,000 steps of adaptive or burn-in steps. The first case is drawing samples from the logarithm of the Rosenbrock function. This case checks the ability of the method to draw samples from a non-linear density function where a naive algorithm can result in the rejection of a lot of samples. The second case is drawing samples from an 8-dimensional correlated Gaussian distribution to check the performance for higher dimensions. Finally, we test the methods for calculating the intensity of multiple source terms in a steady-state groundwater flow problem. The Multivariate Effective Sample Size (mESS, see \cite{ravi-bib-mess}) is used to compare the quality of samples drawn from each method. In many real-world applications, the evaluation of a surrogate is infinitely cheap as compared to the computation of the high-fidelity function. So, we can ignore the evaluation of the surrogate in calculating the computational cost. We compare mESS with respect to the number of high-fidelity evaluations. A higher value of mESS tantamounts to a better method. We also consider the high-fidelity function evaluations done in the offline and burn-in phases in the plots. To create the multi-fidelity surrogate, we start with some random points. We keep on adding points to the training set at the locations of highest variance until the mean squared error of the surrogate is of the order of $10^{-3}$.

\subsection{Rosenbrock function}
We first test MFNUTS on a well-known benchmark test case. We take a two-fidelity scenario where the likelihood of the high-fidelity term $\calL_2$ is the Rosenbrock function and the likelihood of the low-fidelity term $\calL_1$ is a slightly modified Rosenbrock function:
\begin{equation}
    \begin{aligned}
        \calL_1(\theta_1, \theta_2) &= - 12(\theta_2 - \theta_{1}^{2} -1)^2 + (\theta_1 -1)^2  \ , \\
        \calL_2(\theta_1, \theta_2) &= - 50(\theta_2 -\theta_{1}^{2})^2 + (\theta_1 - 1)^2 \ . 
    \end{aligned}  
    \label{ravi-eq:rosenbrock-likelihood}
\end{equation}

% Describe how to build the MFGP surrogate
The unnormalized density functions ($\pi_1(\theta), \pi_2(\theta)$) are exponential of the corresponding likelihood function. The contour lines of the low- and high-fidelity density function is visualized in \RefFig{ravi-fig:rosenbrock-contour} and are considerably different. If one directly uses the low-fidelity function as a guide to drawing proposals for the high-fidelity function then it will lead to a lot of rejections. The transformation from the low-fidelity function to the high-fidelity function involves a translation and a non-linear shape modification. A linear function will not be sufficient to learn this transformation. Therefore, we use the non-linear transformation described in \RefSec{ravi-subsec:mfgp}. The surrogate is created using $50$ high-fidelity function evaluations and $200$ low-fidelity function evaluations. In this case, GPDF generates the smallest mean squared error. %So, we use GPDF surrogate in further steps.

% \begin{figure}[t]
%         \centering
%         \includegraphics[width=0.5\linewidth]{ravi-figures/lf_hf_rosenbrock.pdf}
%         \caption{Contour of low and high-fidelity functions}
%         \label{ravi-fig:rosenbrock-contour}
% \end{figure}
\begin{figure}
    \begin{subfigure}{.5\textwidth}
            \centering
            \includegraphics[width=\linewidth]{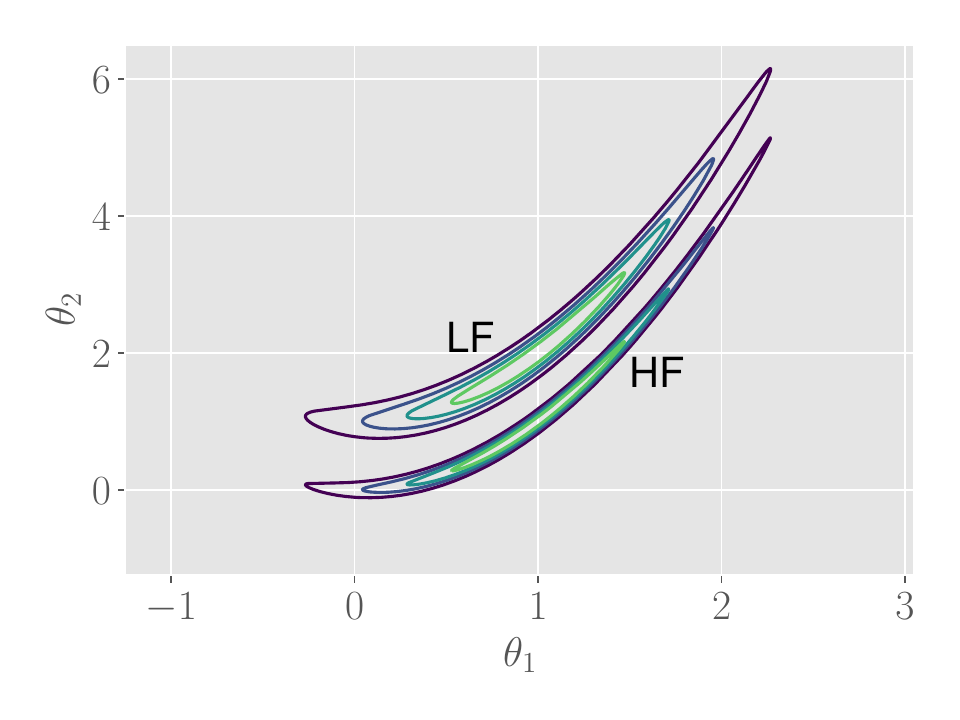}
            \caption{Contour of low and high-fidelity functions}
            \label{ravi-fig:rosenbrock-contour}
        \end{subfigure}
    \begin{subfigure}{.5\textwidth}
        \centering
        \includegraphics[width=\linewidth]{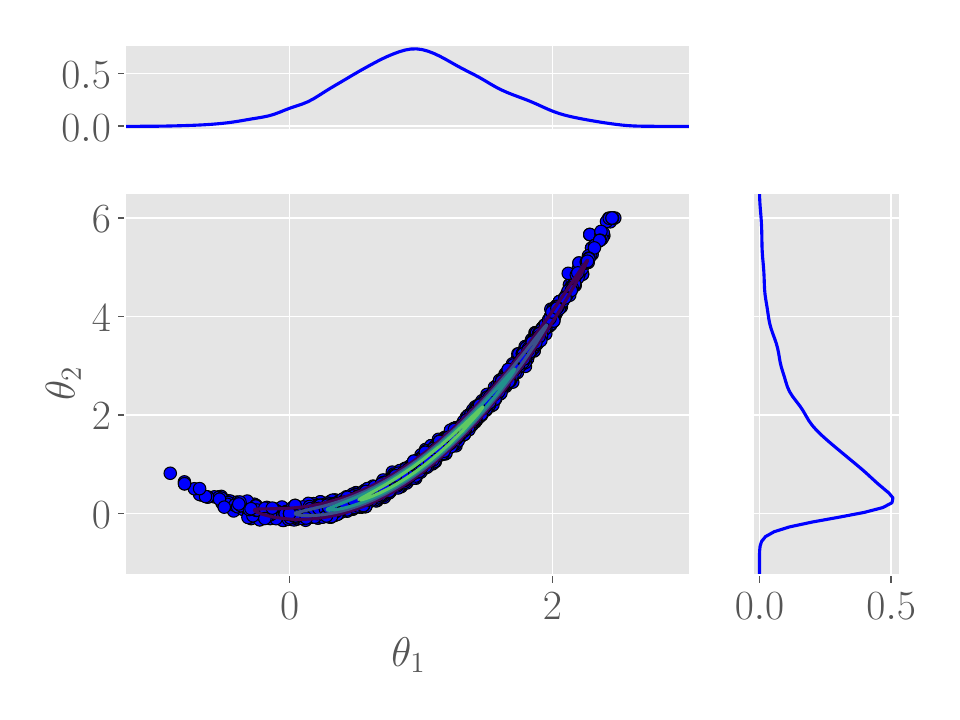}
        \caption{MFNUTS}
        \label{ravi-fig:samples-mfnuts-rosenbrock}
    \end{subfigure} 
    \newline
    \begin{subfigure}{.5\textwidth}
        \centering
        \includegraphics[width=\linewidth]{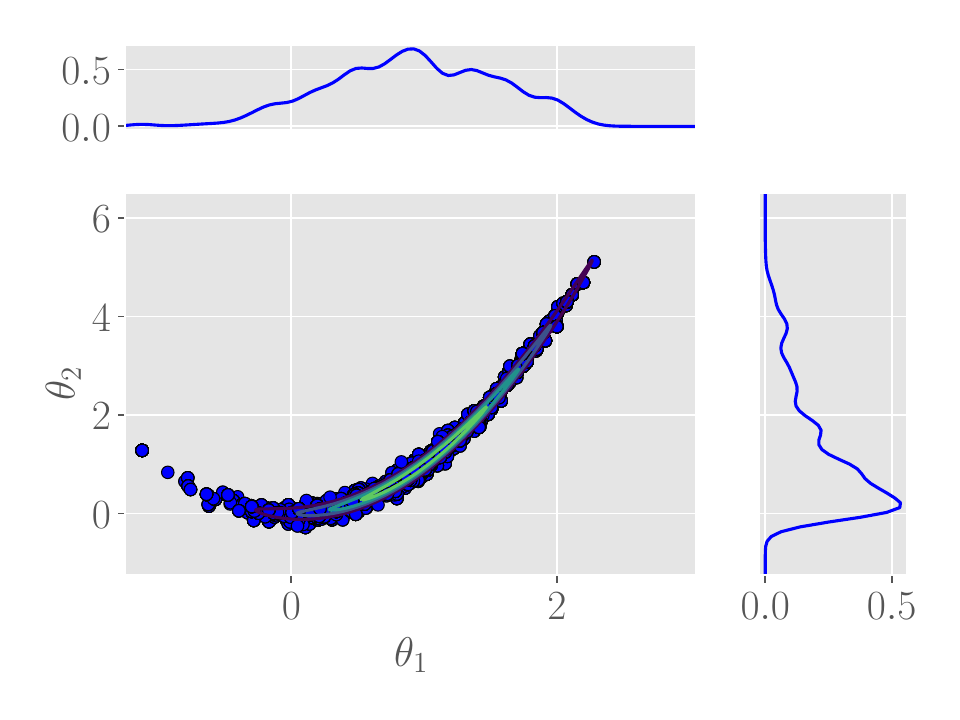}
        \caption{Metropolis-Hastings algorithm}
        \label{ravi-fig:samples-mh-rosenbrock}
    \end{subfigure}
    \begin{subfigure}{.5\textwidth}
        \centering
        \includegraphics[width=\linewidth]{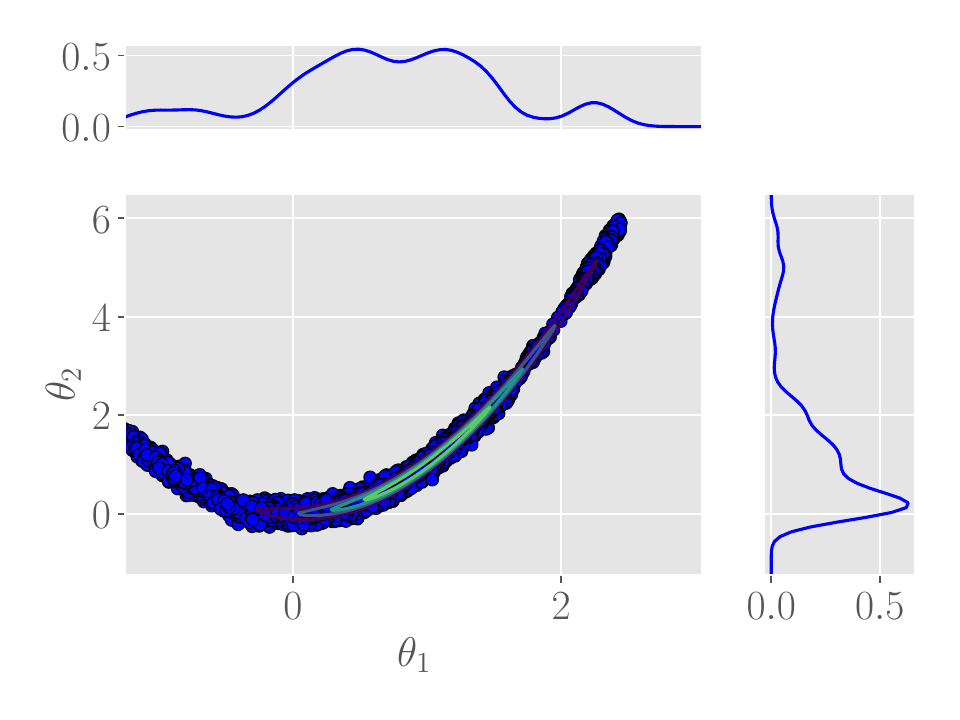}
        \caption{HMC}
        \label{ravi-fig:samples-hmc-rosenbrock}
    \end{subfigure} 
    \newline
    \begin{subfigure}{.5\textwidth}
        \centering
        \includegraphics[width=\linewidth]{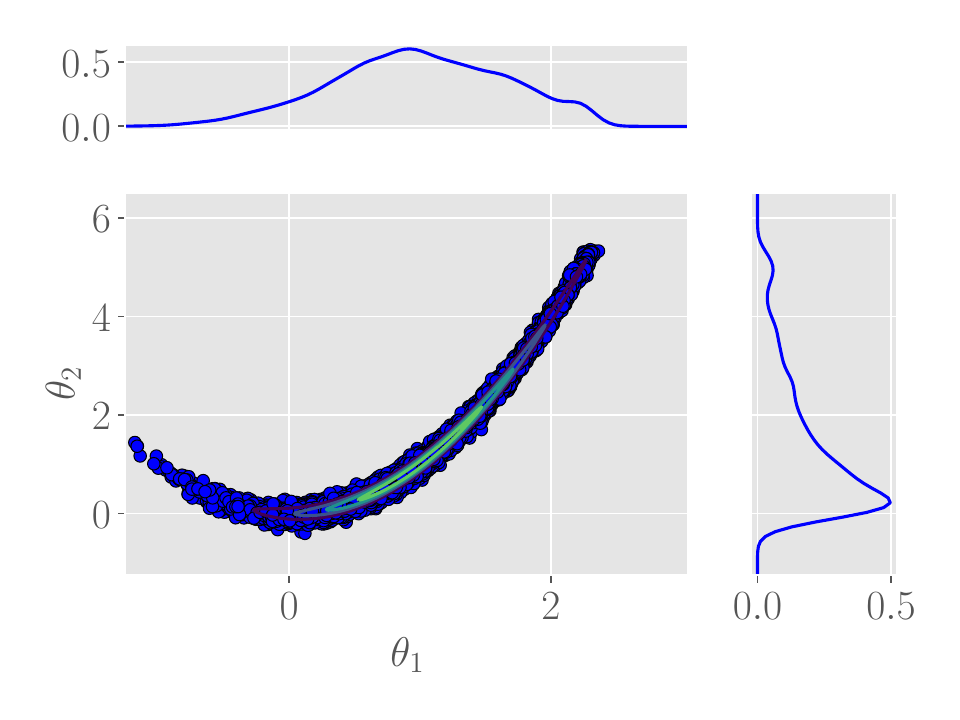}
        \caption{NUTS}
        \label{ravi-fig:samples-nuts-rosenbrock}
    \end{subfigure}
    \begin{subfigure}{.5\textwidth}
        \centering
        \includegraphics[width=\linewidth]{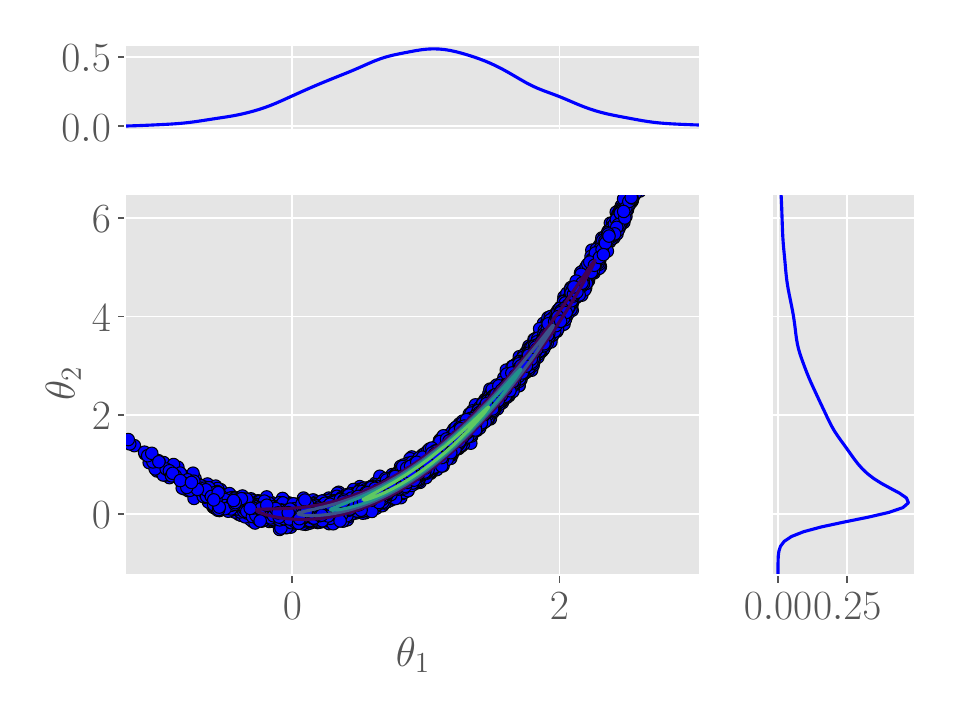}
        \caption{DRAM}
        \label{ravi-fig:samples-dram-rosenbrock}
    \end{subfigure}
    \caption{Contour of low and high-fidelity functions the samples drawn from Rosenbrock function using different algorithms}
    \label{ravi-fig:results-rosenbrock}
\end{figure}

\begin{figure}[!htp]
    \centering
    \includegraphics[width=0.75\linewidth]{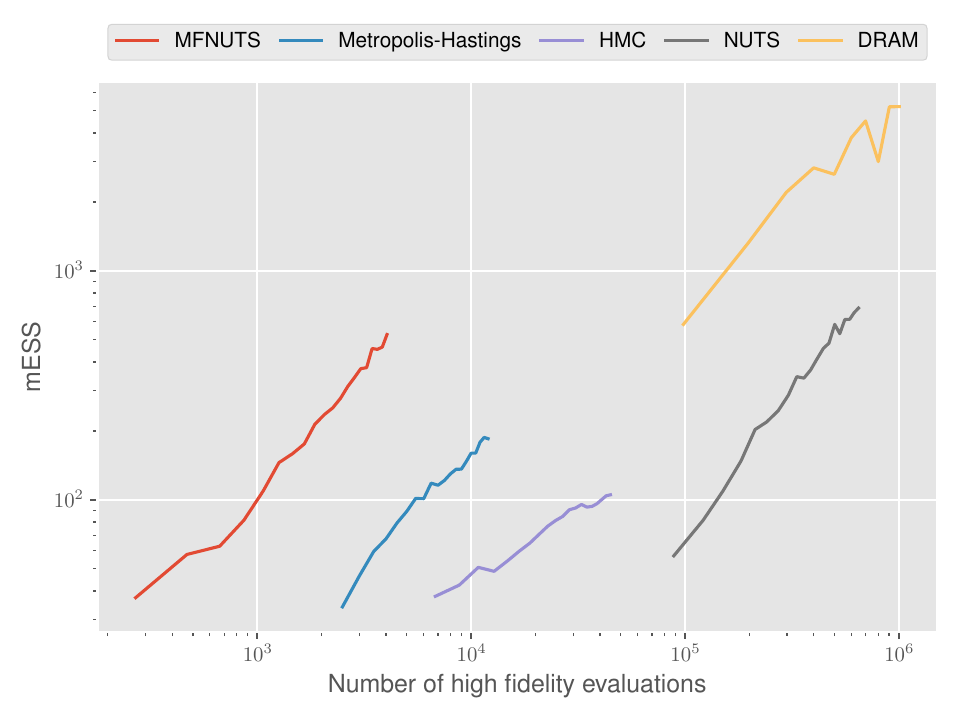}
    \caption{mESS over the number of high-fidelity evaluations for Rosenbrock function}
    \label{ravi-fig:mess-rosenbrock}
\end{figure}

%Compare the quality of samples
\RefFig{ravi-fig:results-rosenbrock} shows that the samples drawn from all four methods represent the target density function. However, HMC and NUTS show a heavy tail as compared to the other methods because the tail regions are generally flat and a small value of the momentum term is not enough for the system to escape the region. MFNUTS has a relatively low number of samples in the tail region as compared to the other gradient-based methods because of the additional acceptance term ($\alpha_{\text{MFNUTS}}$) which mimics the acceptance ratio of the Metropolis-Hastings algorithm. 

%Compare mESS vs number of high-fidelity evaluations
We can observe from \RefFig{ravi-fig:mess-rosenbrock} that MFNUTS outperforms all other methods. The mESS of the samples generated by the Metropolis-Hasting algorithm is worse than the one of NUTS and MFNUTS. This is because the Gaussian density function which is used to propose samples is not representative enough to model the complicated banana-shaped target density. Thus, a lot of samples are rejected leading to a repetition of states which lower the value of mESS.  
HMC has the worst performance amongst all the methods because we used the default value of the number of time steps for the leapfrog integration as provided by Tensorflow Probability instead of manually tuning it. NUTS generates a higher mESS value than HMC. This example shows the importance of the automatic selection of the number of steps. However, NUTS evaluated the model more frequently than HMC. Moreover, a lot of high-fidelity evaluations was also done during the adaptivity steps.
DRAM has the highest mESS as compared to other methods but it requires a lot of high-fidelity function evaluations. Furthermore, DRAM needs extra function evaluations to learn the adaptive transition probability and propose another sample when the previously proposed sample is rejected. MFNUTS outperforms all methods by taking the good side of NUTS and replacing the high-fidelity function evaluation for derivative-evaluation with the surrogate. In most real-world scenarios, the cost of evaluating the surrogate is very small as compared to the one of the high-fidelity function. So, we circumvent the computationally expensive part by using the surrogate.

\subsection{8-d correlated Gaussian distribution}
In this section, we compare the performance of MFNUTS in a higher dimensional space. The low-fidelity function is a Gaussian distribution with zero mean and identity as the covariance matrix. The high-fidelity density function is also a Gaussian distribution with zero mean but now with a tridiagonal matrix as covariance. The log-likelihood of the low and high-fidelity density will be a sphere and an ellipsoid in 8-d space respectively. We use NARGP to learn the transformation between the log-likelihoods using $100$ high-fidelity and $500$ low-fidelity evaluations.  The evolution of mESS with respect to the number of high-fidelity evaluations is shown in \RefFig{ravi-fig:mess-8d-Gaussian}.
\begin{figure}[t]
    \centering
    \includegraphics[width=.75\linewidth]{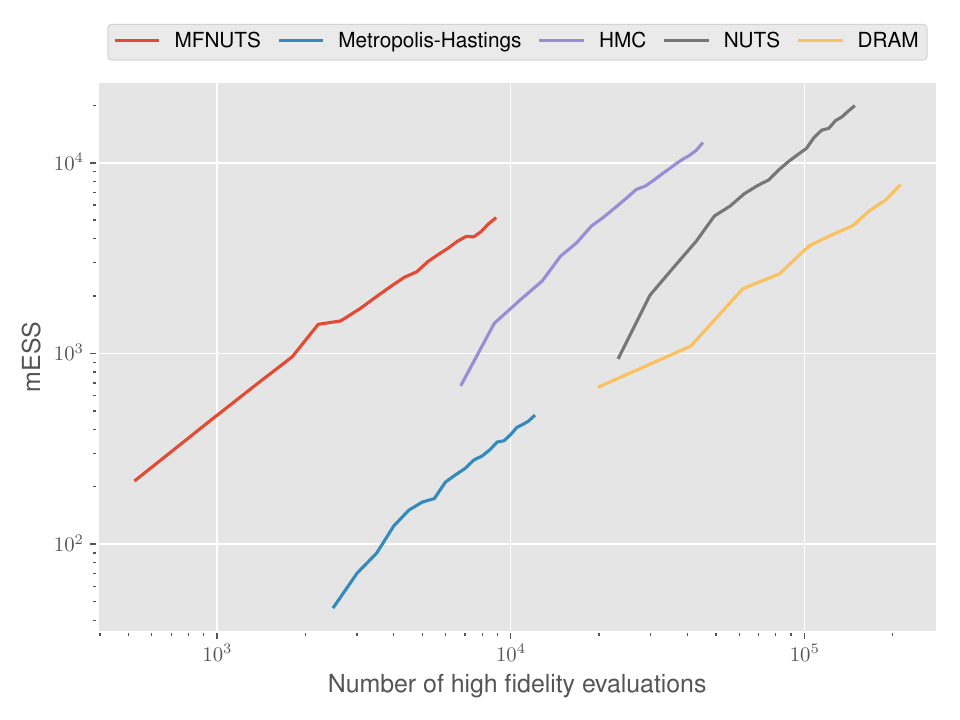}
    \caption{mESS over the number of high-fidelity evaluations for 8-d Gaussian test case.}
    \label{ravi-fig:mess-8d-Gaussian}
\end{figure} 
We again observe that MFNUTS outperforms the other methods. We also observe that the gradient-based algorithms (HMC and NUTS) and DRAM outperform the Metropolis-Hastings algorithm which is the general trend observed in high-dimensional sampling problems, since they provide a better proposal than the Metropolis-Hastings algorithm. We also observe that HMC has higher mESS for the same number of high-fidelity evaluations. We assume that the default value of the number of integration steps was suitable for the target distribution. As in the previous example, MFNUTS is better than other methods by delegating the computationally expensive part of NUTS to the surrogate.
\vspace{-1.5em}
\subsection{Steady-state groundwater flow}
Let us consider a two-dimensional steady-state groundwater flow problem with source terms. The governing equation is:
\begin{equation}
    \begin{aligned}
        \frac{\partial}{\partial X}\left( \kappa(X) \frac{\partial u}{\partial X} \right) = S(X) \quad \quad \quad  X \in \Omega
    \end{aligned}
    \label{ravi-eq:groundawter-flow-eq}
\end{equation} 
where, $\Omega:= [0,1]^2$ represents the spatial domain, $\kappa (X) $ represents the diffusion coefficient and $S(X)$ represents the source term. We consider zero Dirichlet boundary conditions.
For the given problem, we assume that the diffusion coefficient is constant $(\kappa (X) = 1)$ and the source term is the summation of $N \in \Integer$ Gaussian sources:
\begin{equation}
    S(X) = \sum_{i=1}^{N} S_i(X) = \sum_{i=1}^{N} \theta_i \calN(\mu_i, \sigma_i^2) \ ,
    \label{ravi-eq:source-term}
\end{equation}
where the $i^{\text{th}}$ Gaussian source is defined by its location $\mu_i$, variance $\sigma_i^2$ and intensity $\theta_i$. 
We consider a case with four sources as shown in \RefFig{ravi-fig:source-poisson}, at locations $\left[ (0.33, 0.33), (0.33, 0.67), (0.67, 0.33), (0.67, 0.67) \right]$ and each with variance $0.01$. We put nine probes marked by red dots in the \RefFig{ravi-fig:source-poisson}. Our goal is to infer the source intensities for some given measurements at the probes. We use the open-source finite element solver FEniCS \cite{ravi-bib-fenics} to solve the differential equation. Measurement data is generated by solving\RefEq{ravi-eq:groundawter-flow-eq} using a mesh size $64 \times 64$ with source intensity $\theta = \left[ 0.75, 1.25, 0.8, 1.2 \right]$ and adding Gaussian noise with variance $0.005$. 

\begin{figure}[t]
    \hspace{-.5cm}
    \begin{subfigure}{.55\textwidth}
        \centering
        \includegraphics[width=\linewidth]{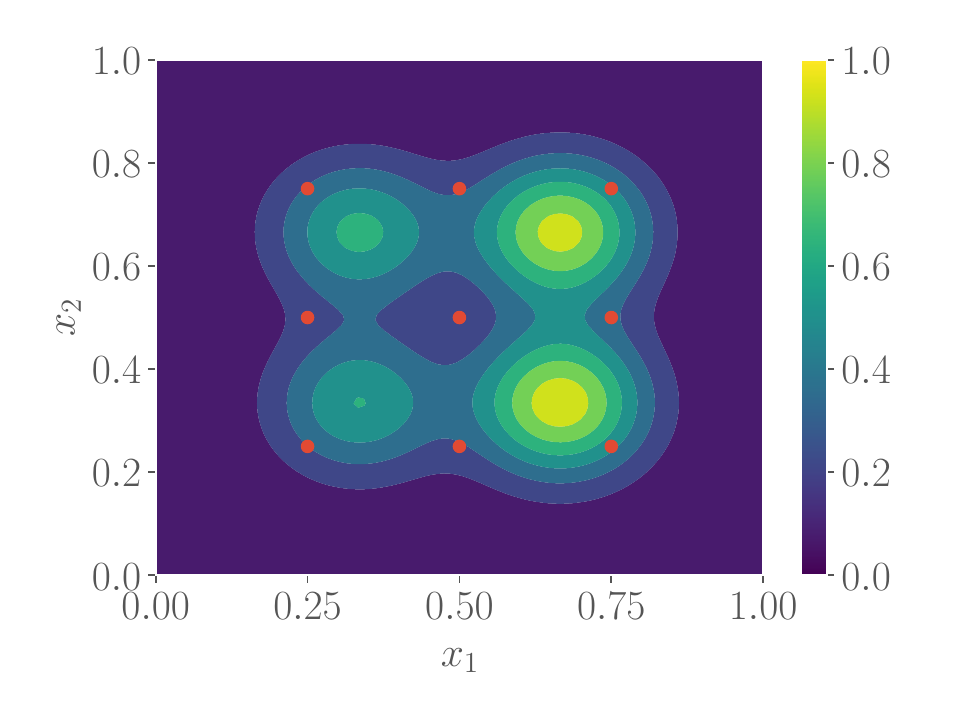}
        \caption{Source term and the probe location}
        \label{ravi-fig:source-poisson}
    \end{subfigure}
    \hspace{-.5cm}
    \begin{subfigure}{.55\textwidth}
        \centering
        \includegraphics[width=\linewidth]{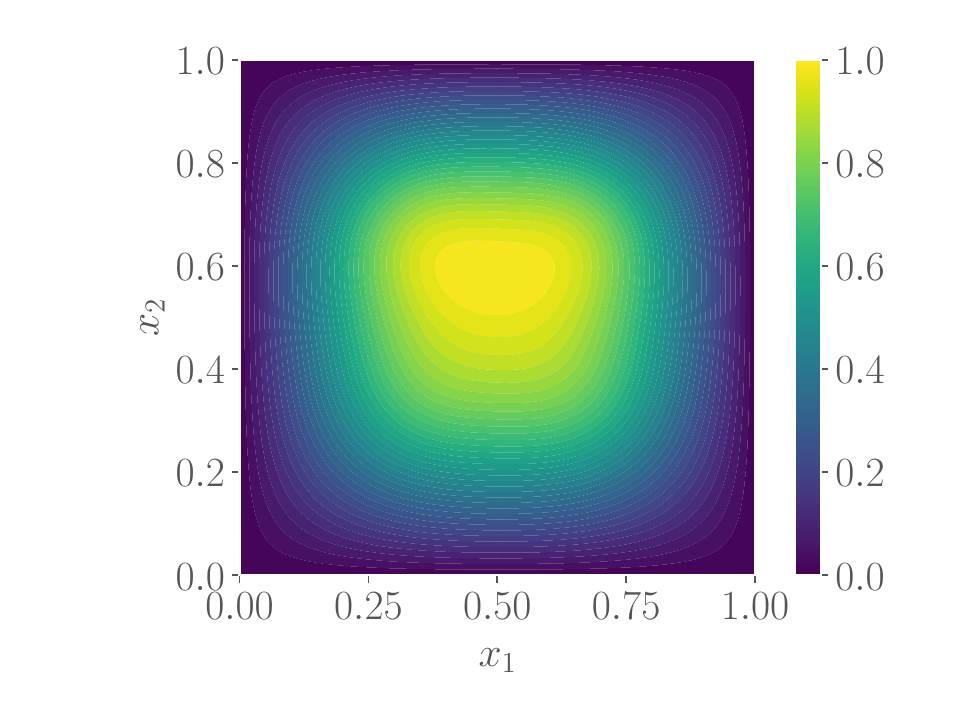}
        \caption{Solution of groundwater flow equation}
        \label{ravi-fig:solution-poisson}
    \end{subfigure} 
    \label{ravi-fig:poisson-setup}
    \caption{Setup of the groundwater flow test case with $\theta = \left[ 0.75, 1.25, 0.8, 1.2 \right]$}
\end{figure}

\begin{figure}
    \centering
    \includegraphics[width=.75\linewidth]{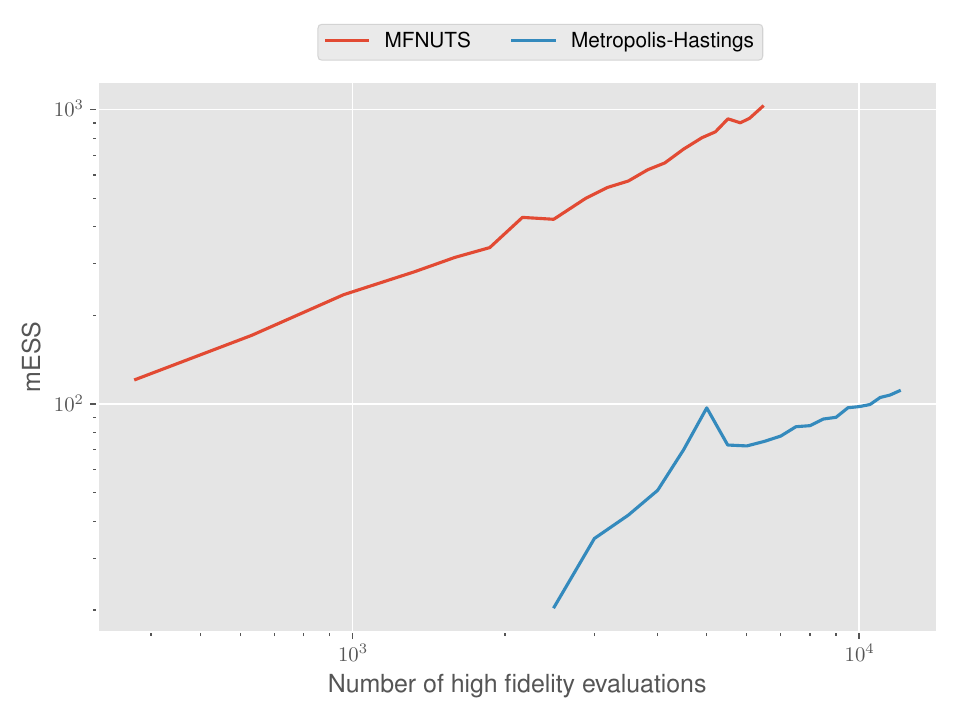}
    \caption{mESS over the number of high-fidelity evaluations for steady-state groundwater flow case.}
    \label{ravi-fig:mess-poisson}
\end{figure} 

\begin{figure}[t!]
    \centering
    \includegraphics[width=\linewidth]{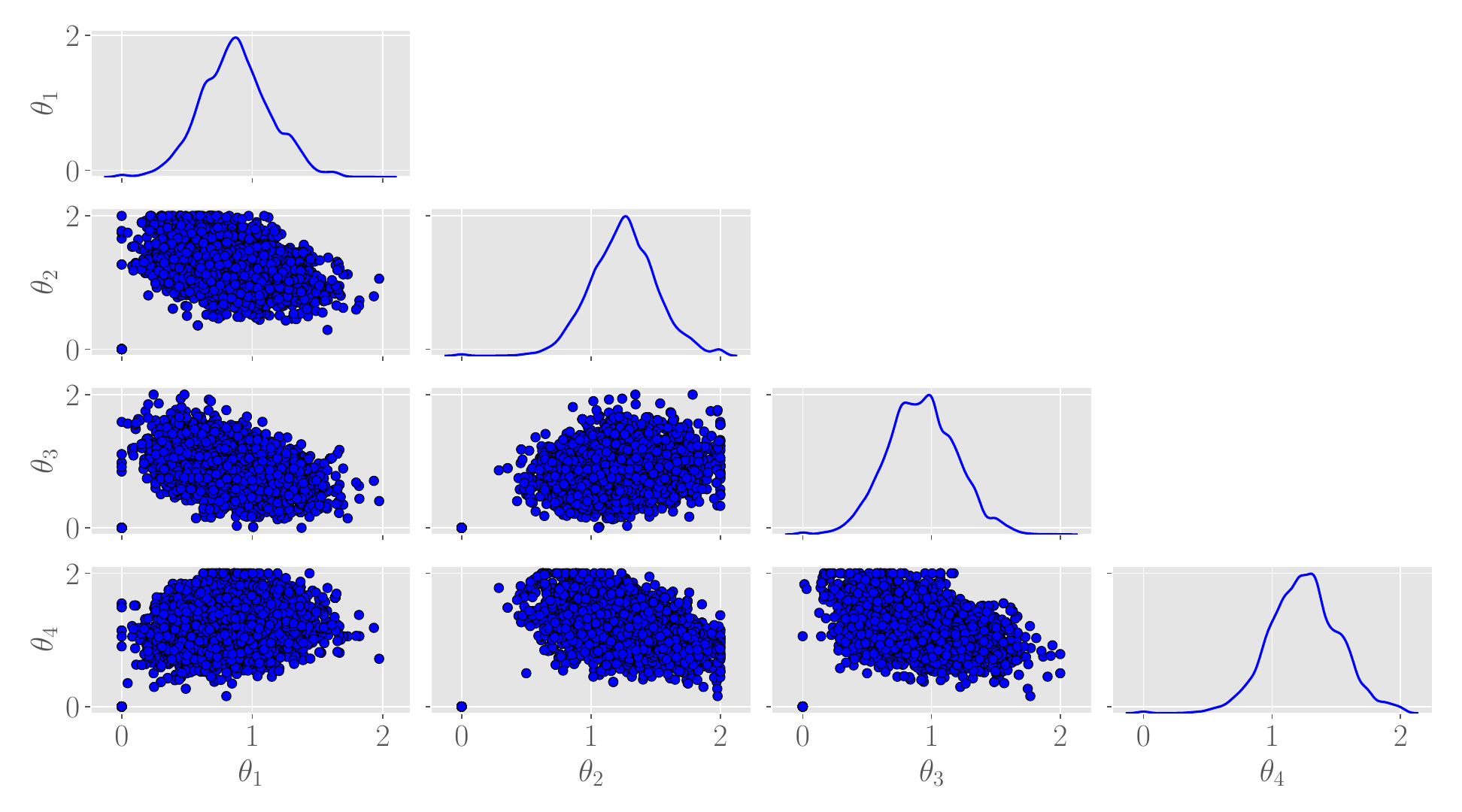}
    \caption{Pairwise plot of the samples of source intensities from MFNUTS for groundwater flow test case}
    \label{ravi-fig:pairwise-poisson-mfnuts}
\end{figure} 

We model the likelihood function as Gaussian with variance $0.005$ (approximately $1\%$ of the mean value of the range $u$) and the prior as a Gaussian with mean $1.0$ and variance $1.0$. Then, we multiply the likelihood with the prior to get the unnormalized posterior distribution which is the target density function. For this test case, we only compare MFNUTS with the Metropolis-Hasting algorithm. The solver does not have support for auto-differentiation. So, the calculation of the derivative can only be done using the finite difference method which is computationally very demanding. So, NUTS and HMC have a clear disadvantage in this test case. We consider two fidelities of solvers for the multi-fidelity sampler. The low and high-fidelity solvers have a mesh size of $8 \times 8$ and $64 \times 64$, respectively. We create separate multi-fidelity surrogates for observations at all the probe locations using NARGP. $70$ high-fidelity and $450$ low-fidelity function evaluations were used for generating the surrogate. Then, we use all the surrogates to compute the posterior which is used as the final multi-fidelity surrogate for the MFNUTS sampler.

We observe from \RefFig{ravi-fig:mess-poisson} that the MFNUTS results in a considerably higher mESS value than the Metropolis-Hastings algorithm for a similar number of high-fidelity function evaluations, as observed in the previous two test cases. We also plot the samples drawn from the MFNUTS in \RefFig{ravi-fig:pairwise-poisson-mfnuts}. The samples are mostly concentrated in  elliptical blobs. The mean of the samples is $[0.87, 1.25, 0.92, 1.25]$ which is close to the source intensity that we used to generate the data. 

\section{Conclusion and future work}
\label{ravi-sec:conclusion}
In this paper, we compare our proposed MFNUTS algorithm with existing single-fidelity sampling methods. In all three cases, MFNUTS outperforms the single-fidelity methods. This was achieved by taking advantage of NUTS and delegating the computationally expensive part to the surrogate. 
The importance of having higher mESS values and a proper exploration of the domain becomes very important when we deal with computationally expensive models. Bad proposals will lead to a waste of computational resources and low mESS will cause a high mean squared error. Our method will be particularly useful in those cases. Moreover, our method also generates samples that are invariant with respect to the high-fidelity model.
However, the quality of the proposal depends upon the surrogate. If the surrogate itself has high error then the proposals will lead to a lot of rejections, thereby decreasing the effective sample size.

It is not essential to use the Gaussian process to build the surrogate. One can use any other method such as a neural network or a sparse grid approximation to build the surrogate. To further improve the algorithm, we can also add the Delayed Rejection \cite{ravi-bib-delayed-rejection} feature. 
%Current implementation of MFNUTS is serial. We can also add parallelization to further decrease the computational time.

\section*{Acknowledgement}

The present contribution is supported by the Helmholtz Association under the research school Munich School for Data Science - MUDS.

% \input{references}
%%%%%%%%%%%%%%%%%%%%%%%% referenc.tex %%%%%%%%%%%%%%%%%%%%%%%%%%%%%%
% sample references
% %
% Use this file as a template for your own input.
%
%%%%%%%%%%%%%%%%%%%%%%%% Springer-Verlag %%%%%%%%%%%%%%%%%%%%%%%%%%
%
% BibTeX users please use
% \bibliographystyle{}
% \bibliography{}

\end{document}